\documentclass[onecolumn]{IEEEtran}
\usepackage{amsmath,amsfonts,amssymb,euscript, graphicx, 
epsfig,
enumerate,float,afterpage, subfigure, ifthen, topcapt, booktabs}%

\usepackage{url}
\usepackage{hyperref}

\newtheorem{thm}{Theorem}

\newtheorem{lem}{Lemma}

\newcommand{\expect}[1]{\mathbb{E}\left[#1\right]}

\newcommand{\norm}[1]{||{#1}||}

\newcommand{\script}[1]{{{\cal{#1} }}}

\newcommand{\transpose}{\top}

\begin{document}

\title
  {Optimal Convergence and Adaptation for Utility Optimal Opportunistic Scheduling}
\author{Michael J. Neely \\ University of Southern California\\ http://www-rcf.usc.edu/$\sim$mjneely\\
}

\markboth{}{Neely}

\maketitle

\begin{abstract} 
This paper considers the fundamental convergence 
time for opportunistic scheduling over time-varying channels. 
The channel state probabilities are unknown and algorithms must
perform some type of estimation and learning while they 
make decisions to optimize network utility.  Existing schemes can 
achieve a utility within $\epsilon$ of optimality, for any desired
$\epsilon>0$,  with convergence
and adaptation times of $O(1/\epsilon^2)$.  This paper shows that 
if the utility function is concave and smooth, then  
$O(\log(1/\epsilon)/\epsilon)$ convergence time is possible
via an existing stochastic variation on the Frank-Wolfe algorithm, 
called the RUN algorithm. 
Next, a converse result is proven to show it is impossible for any algorithm 
to have convergence time better than $O(1/\epsilon)$, provided the algorithm
has no a-priori knowledge of channel state probabilities.  Hence, RUN is 
within a logarithmic factor of convergence time optimality. 
However, RUN has a vanishing stepsize and hence has an 
infinite adaptation time.  Using stochastic Frank-Wolfe with a fixed stepsize
yields improved $O(1/\epsilon^2)$ adaptation 
time, but convergence time increases to $O(1/\epsilon^2)$, similar
to existing drift-plus-penalty based algorithms.  This raises important open questions
regarding optimal adaptation. 
\end{abstract}

\section{Formulation} 

This paper treats opportunistic scheduling for multiple wireless users. 
Consider a wireless system with $n$ users that transmit over their own links.  
The
system operates over slotted time $t \in \{0, 1, 2, \ldots\}$.  The wireless channels can change over time and this affects the 
set of transmission rates available for scheduling.  Specifically, let $\{S[t]\}_{t=0}^{\infty}$ be a process of independent and identically distributed (i.i.d.) \emph{channel state vectors} that take values in some set $\script{S} \subseteq \mathbb{R}^m$, where $m$ is a positive integer.\footnote{The value $m$ can be different from $n$ in the case when the number of channel state parameters is different from the number of links, such as for multi-antenna or multi-subband systems where each link consists of multiple channels.} 
The channel vectors have a probability distribution function $F_S(s) = P[S[t] \leq s]$ for all $s \in \mathbb{R}^m$. However, this distribution function is unknown. 
Every slot $t$, the network controller observes the current $S[t]$ and chooses a \emph{transmission rate vector} $\mu[t] = (\mu_1[t], \ldots, \mu_n[t])$ from 
a set $\Gamma_{S[t]}$.  That is, the set $\Gamma_{S[t]}$ of transmission rate vectors available on slot $t$ depends on the observed  $S[t]$. This is called \emph{opportunistic scheduling} because the network controller can choose to transmit with larger rates on links with currently good channel conditions. The set $\Gamma_{S[t]}$ is typically nonconvex (for example, it might have only a finite number of points). 
 It is assumed that $\Gamma_{S[t]} \subseteq \script{B}$ for all $t \in \{0, 1, 2, \ldots\}$, 
where $\script{B}$ is a bounded $n$-dimensional box  within 
$\mathbb{R}^n$.

For each integer $T>0$, define the time average transmission rate vector $\overline{\mu}[T]$ by: 
$$ \overline{\mu}[T] = \frac{1}{T}\sum_{t=0}^{T-1} \mu[t] $$
The goal is to make decisions over time to maximize the limiting \emph{network utility}:\footnote{The $\liminf$ is used to formally allow algorithms that do not necessarily have regular limits. It represents the smallest possible limiting value over any convergent 
subsequence.  The $\liminf$ always exists   
and is the same as a regular limit whenever the regular limit exists.} 
\begin{align}
\mbox{Maximize:} \quad &  \liminf_{T\rightarrow\infty} \phi(\expect{\overline{\mu}[T]}) \label{eq:p1}  \\
\mbox{Subject to:} \quad & \mu[t] \in \Gamma_{S[t]} \quad, \forall t \in \{0, 1, 2, \ldots\}  \label{eq:p2} 
\end{align}
where $\phi:\script{B}\rightarrow\mathbb{R}$ is a concave \emph{network utility function}  
that is entrywise nondecreasing.  The expectation in the above problem is with respect to the random channel state vectors and
the potentially randomized decision rule for choosing $\mu[t] \in \Gamma_{S[t]}$ on each slot $t$. The above problem 
is particularly challenging because the channel 
state distribution function $F_S$ is unknown.  Algorithms designed without knowledge of $F_S$ are called \emph{statistics-unaware} algorithms. 

This paper considers the \emph{convergence time} required for a statistics-unaware algorithm  
to come within an $\epsilon$-approximation of the optimal utility, where optimality considers 
all algorithms, including those with perfect knowledge of $F_S$.   It is shown that no statistics-unaware algorithm 
can guarantee an $\epsilon$-approximation with convergence time faster than $O(1/\epsilon)$.  Further, it is shown 
that a variation on the Frank-Wolfe algorithm with a running average, 
called RUN, achieves this convergence bound to within a logarithmic factor.  However, this performance holds when starting the time averages at time 0 and using a vanishing stepsize. This raises important questions of \emph{adaptation} over arbitrary intervals of time. 

Problem \eqref{eq:p1}-\eqref{eq:p2} is also important in the special case when there is no time variation so that $\mu[t]$ is chosen every slot from the same fixed set $\Gamma$ (where $\Gamma$ is possibly nonconvex). In this special case, the algorithms considered here 
allow computation of the fractions of time to choose different points in $\Gamma$ to ensure an $\epsilon$-approximation to optimal utility. 

\subsection{Convergence and adaptation definitions} \label{section:convergence-def} 

Define $\phi^{opt}$ as the optimal utility value for problem \eqref{eq:p1}-\eqref{eq:p2}. Fix $\epsilon>0$.  An algorithm is said to 
achieve an \emph{$\epsilon$-approximation with convergence time $C$} if: 
$$ \phi(\expect{\overline{\mu}[T]}) \geq \phi^{opt} - \epsilon \quad, \forall T \geq C $$
An algorithm is said to achieve an \emph{$O(\epsilon)$-approximation with convergence time $O(C)$} if the above holds with $\epsilon$ and $C$ replaced by  constant multiples of $\epsilon$ and $C$.  

Convergence time only considers behavior starting from slot $t=0$.  It is important to consider behavior over \emph{any} interval of time that starts at some arbitrary time $t_0$.  This is important if  the channel state probability distribution $F_S$ changes to a different one at time $t_0$. An algorithm is said to achieve an \emph{$\epsilon$-approximation with adaptation time $C$} if: 
$$ \phi\left(\frac{1}{T}\sum_{t=t_0}^{t_0+T-1} \expect{\mu[t]}\right) \geq \phi^{opt} - \epsilon \quad, \forall t_0 \in \{0, 1, 2, \ldots\}, \forall T \geq C $$
and whenever the channel state distribution function $F_S$ is the same for all time $t \geq t_0$ (the distribution function 
might be different before slot $t_0$).  
This definition captures how long it takes an algorithm to  respond to an unexpected change in channel probabilities that occurs at some time $t_0$.  If the controller knows when such a change occurs, it can simply reset the algorithm by defining the current time as time $0$.  However, the difficulty is that the controller does not necessarily know when a change occurs, and so it cannot reset at appropriate times.  Thus, the adaptation time of an algorithm can be much larger than its convergence time. 

A key aspect of these definitions is that the probability distribution for the system is unknown.  If the distribution were known, one could define a randomized algorithm that transmits with optimized conditional probabilities (given the observed $S[t]$), and convergence of the expectation is immediate.  
An alternative sample-path definition of convergence time is considered in \cite{atilla-convergence-ton}. That work shows
the sample path time average of an integer sequence that converges to an optimal non-integer value must have 
error that decays like $\Omega(1/t)$ (for example, the error might be $1/t$ on odd slots and $-1/t$ on even slots).  This 
holds regardless of whether or not probabilities are known. Of course, if probabilities are known, one can design a randomized algorithm that has optimal expectations on every slot.  
This paper proves that, if probabilities are \emph{unknown}, then even the \emph{expectations} must have an 
$\Omega(1/t)$ utility optimality gap.

\subsection{Prior drift-based algorithm} 

It is known that the  \emph{drift-plus-penalty algorithm} (DPP) of \cite{neely-fairness-ton}\cite{sno-text} achieves an $\epsilon$-approximation with convergence time and adaptation time both being $O(1/\epsilon^2)$.  This algorithm 
operates by defining, for each $i \in \{1, \ldots, n\}$, an auxiliary flow control process $\gamma_i[t]$ and  \emph{virtual queue} $Q_i[t]$ with update equation: 
\begin{equation} \label{eq:q-update} 
 Q_i[t+1] = \max[Q_i[t] + \gamma_i[t] - \mu_i[t], 0] 
 \end{equation} 
The initial condition is typically $Q_i[0]=0$. Every slot $t \in \{0, 1, 2, \ldots\}$, DPP observes $S[t]$ and 
chooses $\mu[t] = (\mu_1[t], \ldots, \mu_n[t])$ and 
 $\gamma[t]=(\gamma_1[t], \ldots, \gamma_n[t])$ via: 
\begin{align}
\mu[t] &= \arg\max_{(r_1[t], \ldots, r_n[t]) \in \Gamma_{S[t]}} \left[\sum_{i=1}^n Q_i[t]r_i[t]\right] \label{eq:mu-DPP}  \\
\gamma[t] &= \arg\max_{(\theta_1[t], \ldots, \theta_n[t]) \in \script{B}} \left[\frac{1}{\epsilon}\phi(\theta_1[t], \ldots, \theta_n[t]) - \sum_{i=1}^n Q_i[t] \theta_i[t]\right] \label{eq:gamma-DPP} 
\end{align}
where $\epsilon>0$ is a parameter that affects a tradeoff between utility optimality and virtual queue size (and hence convergence time). 
This separates the transmission rate decisions $\mu[t]$ according to the (possibly nonconvex)  max-weight rule \eqref{eq:mu-DPP} (which acts only on the queues), and the flow decisions $\gamma[t]$ according to the  (convex) problem \eqref{eq:gamma-DPP} (which uses both the queues and the utility function $\phi$). 
This algorithm is \emph{statistics-unaware}. 
Under a mild \emph{bounded subgradient} condition on the utility function $\phi$, it is shown in \cite{sno-text} that the worst-case virtual queue size is $O(1/\epsilon)$ and the utility achieved over the first $T$ slots satisfies:\footnote{Note that $Q_i[T]/T$ bounds the 
deviation between input flow rate and delivery rate in virtual queue $i$. The worst-case value of $Q_i[T]/T$ is $O(1/\epsilon)/T$, which is $O(\epsilon)$ whenever $T\geq 1/\epsilon^2$. This leads to $1/\epsilon^2$ convergence time \cite{simple-convergence-time-arxiv}.}  
$$ \expect{\phi(\overline{\mu}[T])}  \geq \phi^{opt} - O(\epsilon) \quad \forall T \geq 1/\epsilon^2 $$ 
The utility function is not required to be differentiable and hence this performance holds for non-smooth problems. 
A similar inequality holds for \emph{any interval of time of duration  $1/\epsilon^2$},  
and so the algorithm has an $O(1/\epsilon^2)$ adaptation time.  These results extend to  allow 
additional time average constraints and queue stability constraints \cite{sno-text}. 

\subsection{Prior gradient-based algorithms} 

Alternative \emph{gradient-based algorithms} are developed in  \cite{prop-fair-down}\cite{vijay-allerton02}.  These algorithms
assume the utility function is differentiable.  Let $\phi'(x)^{\transpose}$ denote the transpose of the derivative of $\phi$ at vector $x=(x_1, \ldots, x_n)$, assumed to be a $1 \times n$ row vector: 
$$ \phi'(x)^{\transpose} = \left[\frac{\partial \phi (x)}{\partial x_1} , \ldots,  \frac{\partial \phi (x)}{\partial x_n}\right]$$ 
The algorithms in \cite{prop-fair-down}\cite{vijay-allerton02} use a max-weight type decision with weights determined by the gradient of the utility function evaluated at the time averaged vector.  Specifically, every slot $t>0$ they choose $\mu[t] \in \Gamma_{S[t]}$ as the maximizer of the following expression: 
\begin{equation} \label{eq:full-time-average} 
\phi'(\tilde{\mu}[t-1])^{\transpose} \mu[t] 
\end{equation} 
where $\tilde{\mu}[t-1]$ represents some type of averaging of the previous transmission rates $\mu[0], \ldots, \mu[t-1]$, 
such as the \emph{running average} $\overline{\mu}[t] = \sum_{\tau=0}^{t-1} \mu[\tau]$ (called the RUN algorithm in this paper), 
or an exponentially smoothed average that shall be precisely defined later (called the EXP algorithm in this paper). 
This can be viewed as a stochastic variation on the Frank-Wolfe algorithm for deterministic convex minimization (see, for example,  \cite{Frank-Wolfe}). 
The analyses in \cite{prop-fair-down}\cite{vijay-allerton02} use fluid limit arguments that make precise performance bounds
difficult to obtain.  This gradient-based approach is extended in \cite{stolyar-greedy}\cite{stolyar-gpd-gen} to include additional queue stability constraints. 
To our knowledge, there are no formal analyses of the convergence time of these algorithms.  An analysis in \cite{sno-text} proves the algorithm produces an $\epsilon$-approximation for more general types of problems with queues, with an $O(1/\epsilon)$ queue size, but the proof requires an (unproven) \emph{convergence assumption} and does not specify what the convergence time might be even if the convergence assumption holds. 

\subsection{Related queue stability methods} 

Related problems of minimizing penalty subject to queue stability constraints are considered in \cite{sno-text}\cite{longbo-learning-lagrange}\cite{neely-energy-convergence-ton}\cite{simple-convergence-time-arxiv} using drift-plus-penalty ideas.  The basic $O(1/\epsilon^2)$ convergence results are in \cite{sno-text}\cite{simple-convergence-time-arxiv}.  An important method in \cite{longbo-learning-lagrange} uses 
a \emph{Lagrange multiplier estimation phase} to reduce convergence time to an $O(1/\epsilon^{1 + 2/3})$ bound.\footnote{The work 
\cite{longbo-learning-lagrange} shows the \emph{transient time} for backlog to come close to a Lagrange multiplier vector is $O(1/\epsilon^{2/3})$.  For transients to be amortized, the total time for averages to be within $\epsilon$ of optimality is $O(1/\epsilon^{1+2/3})$.} The work \cite{neely-energy-convergence-ton} treats the special case of average power minimization subject to stability in a simple 1-queue system and shows that convergence time is $O(\log(1/\epsilon)/\epsilon)$.  Recent work in 
\cite{hao-fast-convex-SIAM} uses drift techniques to show that convergence time for  dual-subgradient methods for 
deterministic convex programs can be improved from $O(1/\epsilon^2)$ to $O(1/\epsilon)$. 

\subsection{Our contributions} 

This paper  shows that, assuming the utility function $\phi$ is smooth and has a Lipschitz continuous gradient, 
the convergence time of RUN 
is $O(\log(1/\epsilon)/\epsilon)$, which is superior to that of the DPP algorithm. To our knowledge, this is the first demonstration that such performance is possible.  Further, we show that no statistics-unaware algorithm can achieve a convergence time faster than $O(1/\epsilon)$, and so RUN is within a logarithmic factor of the optimal convergence time.  In the special case when the utility function satisfies an additional \emph{strongly convex} assumption, it is shown that mean square error between the achieved rate vector under RUN and the optimal rate vector decays like $O(\log(t)/t)$, where $t$ is the number of time steps.

Unfortunately, the RUN algorithm uses a vanishing stepsize and 
has no adaptation capabilities. Indeed, it uses a time average starting from time $t=0$ and it  cannot adapt if the probability distribution 
changes halfway through implementation. For example,  if a time average is built over the first $10^3$ slots, and then the probability
distribution changes, it may take $10^6$ slots to amortize the affects of the old and irrelevant time average before the system produces new averages that are close to that desired for the new probability distribution.  
That is, the time required to ``un-average'' an old time average 
can be much longer than the time spent building up this old average. The result is that, if such a change occurs, the 
network utility produced after the change is typically far from optimality.  Formally, it can be shown that the adaptation time, as defined in 
Section \ref{section:convergence-def}, is $\infty$ because the change in probability distribution  can occur at arbitrarily large times $t_0$. 

A simple fix to this adaptability issue is to replace the full time average $\overline{\mu}[t-1]$ used in \eqref{eq:full-time-average}, which averages over the always-growing time interval $\{0, 1, \ldots, t-1\}$,  
with an \emph{exponentially weighted average} (this gives rise to the EXP algorithm).  Fluid model properties of the 
EXP algorithm are considered in \cite{prop-fair-down}\cite{stolyar-greedy}\cite{stolyar-gpd-gen}.  In this paper, we show EXP produces an $O(\epsilon)$ approximation and compute its convergence time.  Unfortunately, while this algorithm has adaptation capabilities similar to the DPP  algorithm, it also has similar $O(1/\epsilon^2)$ convergence time. An open question is whether or not it is possible for both convergence and adaptation times to be improved beyond $O(1/\epsilon^2)$. 

A special case of our stochastic system is a deterministic system where $\mu[t]$ is chosen every slot from a fixed 
set $\Gamma$ that never changes.  When $\Gamma$ is nonconvex, optimal utility typically requires different points
of $\Gamma$ to be selected with different fractions of time. Our results allow computation of fractions of time over which 
the resulting utility is within $\epsilon$ of optimality.  In this context,  a different stepsize
rule is considered that is different from the RUN and EXP algorithms and that relates to classical deterministic convex minimization via Frank-Wolfe.   This stepsize allows fractions of time to be computed with utility error that decays
like $O(1/t)$, faster than the $O(\log(t)/t)$ decay of RUN.

\section{Preliminaries}

\subsection{Assumptions} \label{section:assumptions} 

The set of all transmission rate vectors  available for scheduling is assumed to be bounded. 
Specifically, define the $n$-dimensional box $\script{B} \subseteq \mathbb{R}^n$ by: 
\begin{equation} \label{eq:box} 
\script{B} = [0, \mu_1^{max}] \times \cdots [0, \mu_n^{max}] 
\end{equation} 
where $\mu_i^{max}>0$ are given maximum transmission rates over each 
link $i \in \{1, \ldots, n\}$. For each channel state vector $s \in \script{S}$, the set of available transmission rate vectors 
$\Gamma_s$ is assumed to be a closed
and bounded subset of $\script{B}$.  The network controller chooses $\mu[t] \in \Gamma_{S[t]}$ on each slot $t$, 
and so $0\leq \mu_i[t] \leq \mu_i^{max}$ for all slots $t$ and all  $i \in \{1, \ldots, n\}$.

Let $\phi:\script{B} \rightarrow\mathbb{R}$ be a concave utility function that is entrywise nondecreasing. The function $\phi$ is assumed to be differentiable and  $G$-smooth, so that the gradients $\phi'(x)$ are $G$-Lipschitz continuous: 
\begin{align*}
\norm{\phi'(x) - \phi'(y)} \leq G\norm{x-y} \quad , \forall x, y \in \script{B}
\end{align*}
where $\norm{x} = \sqrt{\sum_{i=1}^n x_i^2}$ denotes the standard Euclidean norm. Formally, the gradients $\phi'(x)$ for points $x$ on the \emph{boundary} of the box $\script{B}$ are defined with respect to limits taken over the interior of the box, and are assumed to 
satisfy the $G$-Lipschitz property above.

An example utility function is
$$ \phi(x) = \sum_{i=1}^n \log(1 + \beta_i x_i) $$ 
where $\beta_i$ are positive values that weight the priority of each user $i \in \{1, \ldots, n\}$. 
Using $\beta_i=\beta$ for all $i$ and choosing a large value of $\beta$ approaches the well known \emph{proportionally fair 
utility} $\sum_{i=1}^n \log(x_i)$.  In this paper, we avoid explicit use of the $\log(x)$ utility because it has a singularity at $x=0$
and is unbounded and has unbounded gradients.

\subsection{Convexity and smoothness} 

It is known that every concave and differentiable function $\phi:\script{B}\rightarrow\mathbb{R}$ 
satisfies the following inequality  \cite{nesterov-book}\cite{bertsekas-convex}: 
\begin{equation} \label{eq:subgradient-inequality} 
\phi(y) \leq \phi(x) + \phi'(x)^{\transpose}(y-x)
\end{equation} 
Further, it is known that every $G$-smooth function $\phi:\script{B}\rightarrow\mathbb{R}$ satisfies the following, often 
called the \emph{descent lemma}  \cite{nesterov-book}\cite{bertsekas-convex}: 
\begin{equation} \label{eq:smooth} 
\phi(y) \geq \phi(x) + \phi'(x)^{\transpose}(y-x) - \frac{G}{2}\norm{y-x}^2 
\end{equation} 

\subsection{The capacity region} 

Let $\Gamma^*$ be the set of all ``one-shot'' expectations $\expect{\mu[0]} \in \mathbb{R}^n$ 
that are possible on slot $0$, considering
all possible conditional probability distributions for choosing $\mu[0] \in \Gamma(S[0])$ in 
reaction to the observed vector $S[0]$.  Since $\mu[0] \in \script{B}$ with probability 1, it follows that the set 
$\Gamma^*$ is in the bounded set $\script{B}$. It can be shown that $\Gamma^*$ is a convex set. 
Define $\overline{\Gamma}^*$ as the closure of $\Gamma^*$.  It can be shown that $\overline{\Gamma}^*$ is 
convex, closed, and bounded.  It is shown in \cite{sno-text} that $\overline{\Gamma}^*$ is the \emph{network capacity region}, in the sense that all possible limiting time average expected transmission rate vectors must lie in the set $\overline{\Gamma}^*$. Further, 
optimality for the problem \eqref{eq:p1}-\eqref{eq:p2} can be defined
by $\overline{\Gamma}^*$. Specifically, define $\phi^{opt}$ as the supremum value of the objective function \eqref{eq:p1} over all possible algorithms. It is known that there exists a vector $x^* \in \overline{\Gamma}^*$ such that $\phi^{opt} = \phi(x^*)$.  In fact, it is shown in \cite{sno-text} that:\footnote{It can similarly be shown that $\phi^{opt}$ is the optimal utility if uniform time averages are replaced by weighted time averages, although that detail is likely not in any publication and shall also be omitted here.} 
\begin{equation} \label{eq:optimality} 
\phi^{opt} = \max_{x \in \overline{\Gamma}^*} \phi(x) 
\end{equation}

\section{Algorithm and analysis} 

This section considers a stochastic version of the deterministic Frank-Wolfe algorithm from \cite{Frank-Wolfe}, also considered
in the fluid limit papers \cite{prop-fair-down}\cite{vijay-allerton02}. 
It is useful to analyze a class of algorithms that use general time-varying weights. Both RUN and EXP have this structure.

\subsection{Weighted averaging algorithms} 

Let $\{\eta_t\}_{t=0}^{\infty}$ be a sequence of real numbers that satisfy $0 < \eta_t \leq 1$ for all $t \in \{0, 1, 2, \ldots\}$. These
shall be used to define a 
sequence of vectors $\gamma[t] \in \mathbb{R}^n$ that are weighted averages of the transmission vectors. 
Specifically, define $\gamma[-1]=0 \in \mathbb{R}^n$, and define: 
\begin{equation} \label{eq:weighted-average} 
\gamma[t] = (1-\eta_t) \gamma[t-1] + \eta_t \mu[t]  \quad, \forall t \in \{0, 1 ,2, \ldots\} 
\end{equation} 
Strictly speaking, the above scheme is an ``approximate'' weighted average of the transmission vectors $\mu[t]$ because it 
initializes $\gamma[-1]$ to 0, rather than to $\mu[0]$.  This ``zero-initialization'' is for convenience later.   Notice that using $\eta_t = \eta$ for all $t$, for a fixed $\eta \in (0,1)$, results in an exponentially weighted average of $\mu[t]$.  Using $\eta_t = 1/(t+1)$ results in a running average of $\mu[t]$. The values $\eta_t$ are often called the \emph{stepsize} on slot $t$.

On each slot $t \in \{0, 1, 2, \ldots\}$, we consider a gradient-based opportunistic scheduling 
algorithm that observes  $\gamma[t-1]$ and the current channel state $S[t]$ and 
chooses the transmission vector $\mu[t]$ to solve: 
\begin{align} 
\mbox{Maximize:} \quad & \phi'(\gamma[t-1])^{\transpose}\mu[t]  \label{eq:max1} \\
\mbox{Subject to:} \quad & \mu[t] \in \Gamma_{S[t]} \label{eq:max2}
\end{align}
The above decision chooses $\mu[t]$ to maximize a linear function over the compact set $\Gamma_{S[t]}$, and so there is at least one maximizer.  If more than one maximizer exists, ties are broken arbitrarily.  Formally, the rule for breaking ties is assumed to be probabilistically measurable, so that $\gamma[t]$ is a valid random variable with well defined expectations that lie in the box $\script{B}$.  

A key property of the above algorithm is the following: If $\mu[t] \in \Gamma_{S[t]} $ is the decision produced 
by the  rule \eqref{eq:max1}-\eqref{eq:max2} on a slot $t$,  then: 
\begin{equation} \label{eq:key} 
\phi'(\gamma[t-1])^{\transpose}\mu[t] \geq \phi'(\gamma[t-1])\mu^*[t] \quad \forall t \in \{0, 1, 2,  \ldots\} 
\end{equation} 
where $\mu^*[t]$ is any other (possibly randomized) decision vector in the set $\Gamma_{S[t]}$.  This holds
simply because $\mu[t]$ is (by definition) the maximizer of \eqref{eq:max1}.  
Two other useful properties that hold for all slots $t \in \{0, 1,2, \ldots\}$ are: 
\begin{align}
\mu[t]-\gamma[t-1] &= \frac{\gamma[t]-\gamma[t-1]}{\eta_t} \label{eq:another1} \\
\phi'(\gamma[t-1])^{\transpose}(\gamma[t]-\gamma[t-1]) &\leq \phi(\gamma[t]) - \phi(\gamma[t-1]) + \frac{G}{2}\norm{\gamma[t]-\gamma[t-1]}^2 \label{eq:another2} 
\end{align}
where \eqref{eq:another1} follows by \eqref{eq:weighted-average};  \eqref{eq:another2} follows by the smoothness property \eqref{eq:smooth}.  

\subsection{Performance lemmas} 

\begin{lem} \label{lem:compare}  For each slot $t \in \{0, 1, 2,  \ldots\}$ the weighted averaging algorithm ensures: 
$$ \expect{\phi'(\gamma[t-1])^{\transpose}(\mu[t]-\gamma[t-1])} \geq \phi^{opt} -  \expect{\phi(\gamma[t-1])} $$
where $\phi^{opt}$ is the optimal objective value for problem \eqref{eq:p1}-\eqref{eq:p2}. 
\end{lem} 
\begin{proof} 
Fix $t\in \{0, 1, 2,  \ldots\}$ and let $\mu[t]$ be the decision made by the weighted averaging algorithm on slot $t$. 
Recall that $\Gamma^*$ is the set of all achievable one-shot expectations $\expect{\mu[0]}$. 
Fix $x \in \Gamma^*$ and let $\mu^*[t] \in \Gamma_{S[t]}$ be a stationary and randomized algorithm that makes decisions as a randomized
function of $S[t]$ to yield $\expect{\mu^*[t]} = x$. Applying inequality \eqref{eq:key} gives:  
$$ \phi'(\gamma[t-1])^{\transpose}\mu[t] \geq \phi'(\gamma[t-1])^{\transpose}\mu^*[t] $$
Taking expectations of this gives
\begin{align}
\expect{ \phi'(\gamma[t-1])^{\transpose}\mu[t]} &\geq \expect{\phi'(\gamma[t-1])^{\transpose}\mu^*[t]} \nonumber \\ 
&\overset{(a)}{=} \expect{\phi'(\gamma[t-1])^{\transpose}} \expect{\mu^*[t]} \nonumber \\ 
&=\expect{\phi'(\gamma[t-1])^{\transpose}} x \label{eq:x-val} 
\end{align}
where equality (a) holds because channel state vectors $S[t]$ are i.i.d. over slots and $\mu^*[t]$ depends only on $S[t]$, so that it is independent of $\gamma[t-1]$.  Inequality \eqref{eq:x-val} holds for all vectors $x \in \Gamma^*$.  Taking a limit as $x \rightarrow x^*$, where $x^*$ is a fixed vector in $\overline{\Gamma}^*$ such that $\phi(x^*)=\phi^{opt}$, gives: 
$$\expect{ \phi'(\gamma[t-1])^{\transpose}\mu[t]} \geq \expect{\phi'(\gamma[t-1])^{\transpose}}x^* $$
Subtracting the same value from both sides of the above inequality gives: 
\begin{equation} \label{eq:manip} 
\expect{\phi'(\gamma[t-1])^{\transpose}(\mu[t]-\gamma[t-1])} \geq \expect{\phi'(\gamma[t-1])^{\transpose}(x^*-\gamma[t-1])} 
\end{equation} 
However, the subgradient inequality \eqref{eq:subgradient-inequality} for concave functions yields: 
$$ \phi'(\gamma[t-1])^{\transpose}(x^*-\gamma[t-1]) \geq \phi(x^*) - \phi(\gamma[t-1]) $$
Taking expectations of the above inequality and substituting into the right-hand-side of \eqref{eq:manip} yields the result. 
\end{proof} 

Define $\mu^{max} = (\mu_1^{max}, \ldots, \mu_n^{max})$.  We have the following lemma. 

\begin{lem} \label{lem:performance}  For all slots $t \in \{0, 1, 2, \ldots\}$ we have: 
\begin{equation} \label{eq:performance} 
\frac{1}{\eta_t} \expect{\phi(\gamma[t])}  \geq \phi^{opt} + \left[\frac{1}{\eta_t}- 1\right]\expect{\phi(\gamma[t-1])} - \frac{\eta_t G\norm{\mu^{max}}^2}{2}  \quad , \forall t \in \{0, 1, 2, \ldots\} 
\end{equation} 
\end{lem} 

\begin{proof} 
By Lemma \ref{lem:compare} we have for all slots $t\in \{0, 1, 2,  \ldots\}$: 
\begin{align}
\expect{\phi(\gamma[t-1])} &\geq \phi^{opt} - \expect{\phi'(\gamma[t-1])^{\transpose}(\mu[t]-\gamma[t-1])}  \nonumber \\
&\overset{(a)}{=}  \phi^{opt} - \frac{1}{\eta_t}\expect{\phi'(\gamma[t-1])^{\transpose}(\gamma[t]-\gamma[t-1])} \nonumber \\
&\overset{(b)}{\geq} \phi^{opt}  - \frac{1}{\eta_t} \expect{\phi(\gamma[t])-\phi(\gamma[t-1]) + \frac{G}{2}\norm{\gamma[t]-\gamma[t-1]}^2 } \nonumber \\
&\overset{(c)}{=} \phi^{opt} + \frac{\expect{\phi(\gamma[t-1])}}{\eta_t}-\frac{\expect{\phi(\gamma[t])}}{\eta_t} - \frac{\eta_t G}{2} \expect{\norm{\mu[t]-\gamma[t-1]}^2} \nonumber \\
&\overset{(d)}{\geq} \phi^{opt} + \frac{\expect{\phi(\gamma[t-1])}}{\eta_t}-\frac{\expect{\phi(\gamma[t])}}{\eta_t}  - \frac{\eta_t G\norm{\mu^{max}}^2}{2}   \label{eq:almost-done} 
\end{align}
where (a) holds by \eqref{eq:another1}; (b) holds by \eqref{eq:another2}; (c) holds by \eqref{eq:another1}; and (d) holds 
because $\mu[t]$ and $\gamma[t-1]$ lie in the box $\script{B}$ and the largest possible magnitude of their difference is $\norm{\mu^{max}}$. 
Rearranging terms yields the result. 
\end{proof} 

The above lemma shall be used to evaluate the EXP and RUN algorithms.

\subsection{The RUN algorithm} 

Let $\eta_t = \frac{1}{t+1}$ for $t \in \{0, 1, 2, \ldots\}$.  With these weights, the iteration \eqref{eq:weighted-average} produces a \emph{running average} of the $\mu[t]$ values: 
$$ \gamma[t] = \frac{t}{t+1} \gamma[t-1] + \frac{1}{t+1}\mu[t]  \implies \gamma[t] = \frac{1}{t+1} \sum_{\tau=0}^{t} \mu[\tau] = \overline{\mu}[t+1] \quad, \forall t \in \{0, 1, 2, \ldots \} $$
This shall be called the RUN algorithm. 

\begin{thm} \label{thm:RUN} Under the RUN algorithm, we have for all integers $T>0$:\footnote{By Jensen's inequality for the concave function $\phi$  we
know $\phi(\expect{\overline{\mu}[T]})\geq \expect{\phi(\overline{\mu}[T])}$, and so Theorems \ref{thm:RUN} and \ref{thm:EXP} also provide bounds on $\phi(\expect{\overline{\mu}[T]})$.} 
$$\expect{\phi\left(\overline{\mu}[T] \right)} \geq \phi^{opt} - \frac{G \norm{\mu^{max}}^2(1+\log(T))}{2T} $$
\end{thm} 

\begin{proof}  
Fix $T>0$ as an integer.  Summing inequality \eqref{eq:performance} over $t \in \{0, 1, \ldots, T-1\}$ gives: 
$$\sum_{t=0}^{T-1} \frac{1}{\eta_t} \expect{\phi(\gamma[t])} \geq T\phi^{opt} + \sum_{t=0}^{T-1}\left[\frac{1}{\eta_t}-1\right] \expect{\phi(\gamma[t-1])} - \frac{G\norm{\mu^{max}}^2}{2} \sum_{t=0}^{T-1}\eta_t $$
Rearranging terms gives
\begin{align*}
\expect{\sum_{t=0}^{T-1}\phi(\gamma[t-1])} 
&\geq T\phi^{opt} + \sum_{t=0}^{T-2} \expect{\phi(\gamma[t])} \left[\frac{-1}{\eta_t}  + \frac{1}{\eta_{t+1}}\right]  \nonumber \\
& + \left[\frac{\expect{\phi(\gamma[-1])}}{\eta_0} - \frac{\expect{\phi(\gamma[T-1])}}{\eta_{T-1}}  \right] 
  - \frac{G \norm{\mu^{max}}^2}{2} \sum_{t=0}^{T-1} \eta_t 
\end{align*}
Substituting $\eta_t= 1/(t+1)$ gives  
\begin{align*}
\expect{\sum_{t=0}^{T-1}\phi(\gamma[t-1])} 
&\geq T\phi^{opt} + \sum_{t=0}^{T-2} \expect{\phi(\gamma[t])} + \expect{\phi(\gamma[-1])} - T\expect{\phi(\gamma[T-1]) } 
  - \frac{G \norm{\mu^{max}}^2}{2} \sum_{t=0}^{T-1} \frac{1}{t+1}  
\end{align*}
Canceling common terms in the above inequality and rearranging yields
\begin{align*}
T\expect{\phi(\gamma[T-1])} &\geq T \phi^{opt}  -  \frac{G \norm{\mu^{max}}^2}{2} \sum_{t=0}^{T-1} \frac{1}{t+1} \\
&\geq T \phi^{opt} - \frac{G\norm{\mu^{max}}^2}{2} (1+\log(T))
\end{align*}
Dividing by $T$ and using the fact that $\gamma[T-1] = \overline{\mu}[T]$ gives the result. 
\end{proof}

This theorem shows that utility converges to the optimal value $\phi^{opt}$ as $T\rightarrow\infty$. Deviation from optimality 
decays like  $\log(T)/T$.  
Fix $\epsilon>0$.  Then we are within $O(\epsilon)$ of optimality after a \emph{convergence time} of $O(\log(1/\epsilon)/\epsilon)$. 

\subsection{The EXP algorithm} 

Fix $\eta \in (0,1)$ and define $\eta_t = \eta$ for all $t \in \{0, 1, 2, \ldots\}$. This shall be called the EXP algorithm. 

\begin{thm} \label{thm:EXP} Under the EXP algorithm, we have for all integers $T>0$: 
$$\expect{\phi\left(\overline{\mu}[T] \right)} \geq \phi^{opt} - \left[\frac{\phi^{opt} - \phi(0)}{\eta T}\right] - \frac{\eta G \norm{\mu^{max}}^2}{2} $$
\end{thm} 

\begin{proof} 
Substituting $\eta_t = \eta$ into \eqref{eq:performance} gives for all $t \in \{0, 1, 2, \ldots\}$, 
$$  \frac{1}{\eta} \expect{\phi(\gamma[t])} \geq \phi^{opt} +\left[ \frac{1}{\eta} - 1\right] \expect{\phi(\gamma[t-1])} - \frac{\eta G\norm{\mu^{max}}^2}{2}   $$
Rearranging terms gives: 
\begin{equation} \label{eq:useful-EXP}
\expect{\phi(\gamma[t-1])} \geq \phi^{opt} + \frac{1}{\eta}\expect{\phi(\gamma[t-1]) - \phi(\gamma[t])}  - \frac{\eta G \norm{\mu^{max}}^2}{2} 
\end{equation} 
Fix $T>0$ as an integer.  Summing over $t \in \{0, 1, \ldots, T-1\}$ gives
\begin{align*}
 \expect{\sum_{t=0}^{T-1} \phi(\gamma[t-1])} &\geq T \phi^{opt} + \frac{\expect{\phi(\gamma[-1])}}{\eta} -\frac{\expect{\phi(\gamma[T-1])}}{\eta} - \frac{G \eta T\norm{\mu^{max}}^2}{2} \\
 &\geq  T \phi^{opt} + \frac{\phi(0)}{\eta} -\frac{\phi^{opt}}{\eta} - \frac{G \eta T\norm{\mu^{max}}^2}{2} 
 \end{align*}
 where the last inequality holds because $\gamma[-1]=0$ with probability 1, and $\expect{\phi(\gamma[T-1])} \leq \phi^{opt}$ (see Lemma \ref{lem:dominate} in the appendix).  Dividing the above inequality by $T$ and using Jensen's inequality on the concave function 
 $\phi$ gives: 
 $$ \expect{\phi\left(\frac{1}{T}\sum_{t=0}^{T-1} \gamma[t-1]\right)} \geq \phi^{opt} - \left[\frac{\phi^{opt} - \phi(0)}{\eta T}\right] - \frac{G \eta \norm{\mu^{max}}^2}{2} $$
It remains to relate the time average of the $\gamma[t-1]$ process to that of the $\mu[t]$ process.  Substituting $\eta_t = \eta$ into  \eqref{eq:another1} and summing over $t \in \{0,  \ldots, T-1\}$ (and dividing by $T$) gives: 
\begin{align*}
\frac{1}{T}\sum_{t=0}^{T-1} \mu[t] =  \frac{1}{T}\sum_{t=0}^{T-1} \gamma[t-1]  + \frac{\gamma[T-1]-\gamma[-1]}{\eta T}  \geq \frac{1}{T}\sum_{t=0}^{T-1} \gamma[t-1]  
\end{align*}
where the final inequality uses the fact that $\gamma[-1] = 0 \leq \gamma[T-1]$. 
\end{proof} 

Fix $\epsilon>0$.  By defining $\eta = \epsilon$, Theorem \ref{thm:EXP} implies that EXP achieves an $O(\epsilon)$-approximation with convergence time $T = 1/\epsilon^2$.   A similar argument can be given that sums \eqref{eq:useful-EXP} over the interval $\{t_0, \ldots, t_0+T-1\}$ to show that the \emph{adaptation time} of EXP is also $1/\epsilon^2$ (this argument is omitted for brevity).  This 
argument works because the stepsize $\eta$ does not change with time, which is not the case for the RUN algorithm.

\subsection{Relation to deterministic Frank-Wolfe} \label{section:unusual}

The analysis of RUN and EXP in the above subsections is similar to the deterministic analysis of the Frank-Wolfe algorithm (see, for example, \cite{Frank-Wolfe}).  An important difference is that the above analysis treats the \emph{stochastic case} and considers performance in terms of the time average $\overline{\mu}[T]$ achieved over time.  In contrast, the classical Frank-Wolfe algorithm seeks a single vector $x$ within a given convex set that is close to optimal, with no regard to how time averages behave.  

It is interesting to note that a modified stepsize $\eta_t = 2/(t+2)$ is used for deterministic convex minimization in  \cite{Frank-Wolfe} to show that an approximate vector $x$ can be computed after $T$ iterations with error bounded by $O(1/T)$ (which is faster than the $O(\log(T)/T)$ results of RUN).  At first glance, this suggests that using the modified stepsize $\eta_t = 2/(t+2)$ in the stochastic problem 
might remove the $\log(\cdot)$ factor.  However, the same
analysis of the deterministic problem cannot be used in our stochastic context.  Intuitively, this is because the stochastic problem seeks  desirable time average behavior as the stochastic algorithm runs, while deterministic Frank-Wolfe desires computation of a single deterministic vector with no regards to time average behavior.  It is not clear if the $\log(\cdot)$ factor can be removed for the stochastic time average problem. 

However, the stepsize rule $\eta_t = 2/(t+2)$ is still useful for stochastic scheduling problems. 
It leads to an algorithm that is different from RUN and EXP.  The resulting $\gamma[t]$ value is an unusual weighted average of $\{\mu[0], \ldots, \mu[t]\}$.  Indeed, using $\eta_t=2/(t+2)$ in \eqref{eq:weighted-average} gives
\begin{align}
\gamma[0] &= \mu[0] \label{eq:unusual1} \\
\gamma[1] &= \frac{1}{3}\mu[0] + \frac{2}{3}\mu[1] \label{eq:unusual2} \\
\gamma[2] &= \frac{1}{6}\mu[0] + \frac{1}{3}\mu[1] + \frac{1}{2}\mu[2] \label{eq:unusual3} \\
\gamma[3] &= \frac{1}{10}\mu[0] + \frac{1}{5}\mu[1] + \frac{3}{10}\mu[2] + \frac{2}{5}\mu[3] \label{eq:unusual4} 
\end{align}
and so on. 
The next theorem shows that the utility associated with this 
unusual weighted average $\gamma[T]$ deviates from $\phi^{opt}$ by $O(1/T)$, 
although this does not hold for the utility associated with the online time average
 transmission rate $\overline{\mu}[T]$.    This unusual weighted average is particularly 
 useful in the offline deterministic contexts described in Section \ref{section:deterministic}. 
 The proof is similar to that of the deterministic case in \cite{Frank-Wolfe} and closely follows
 that proof structure.

\begin{thm} \label{thm:new-stepsize}  Using algorithm \eqref{eq:max1}-\eqref{eq:max2} with stepsize $\eta_t = 2/(t+2)$ yieds: 
$$ \expect{\phi(\gamma[t])} \geq \phi^{opt} - \frac{2G\norm{\mu^{max}}^2}{t+1} \quad, \forall t \in \{0, 1, 2, \ldots\} $$
\end{thm} 
\begin{proof} 
Define $c = G\norm{\mu^{max}}^2/2$.  Substituting $c$ into Lemma \ref{lem:performance} and multiplying both sides by $\eta_t$
gives
 $$ \expect{\phi(\gamma[t])} \geq \eta_t\phi^{opt} + (1-\eta_t)\expect{\phi(\gamma[t-1])} - c\eta_t^2 $$
 which holds for all $t \in \{0, 1, 2,\ldots\}$. 
 Define $a_t = \phi^{opt} - \expect{\phi(\gamma[t])}$.  Multiplying the above inequality by $-1$ and adding $\phi^{opt}$ to both sides gives: 
 $$ a_t \leq (1-\eta_t)a_{t-1} + c\eta_t^2 $$
 Substituting $\eta_t = 2/(t+2)$ gives
 \begin{equation} \label{eq:UNUSUAL} 
  a_t \leq \frac{t}{t+2}a_{t-1} + \frac{4c}{(t+2)^2}  \quad, \forall t \in \{0, 1, 2, \ldots\} 
  \end{equation} 
 It follows that $a_0 \leq c \leq 4c$. 
 Suppose there is an integer $k>0$ such that 
 $a_t \leq \frac{4c}{t+1}$ for all $t \in \{0, \ldots, k-1\}$ (it holds for $k-1=0$).  We prove 
 it also holds for time $t=k$. We have by \eqref{eq:UNUSUAL} together with $a_{k-1} \leq 4c/k$: 
\begin{align*}
 a_{k} &\leq \left(\frac{k}{k+2}\right)\frac{4c}{k} + \frac{4c}{(k+2)^2} \\
 &= \frac{4c(k+3)}{(k+2)^2} \\
 &\leq \frac{4c}{k+1} 
\end{align*} 
where the final inequality holds because $\frac{k+3}{(k+2)^2} \leq \frac{1}{k+1}$ for all positive integers $k$. 
By induction, it follows that $a_t \leq 4c/(t+1)$ for all $t \in \{0, 1, 2,\ldots\}$, which proves the result. 
\end{proof}

\subsection{Strongly concave utility functions} 

Consider again the RUN algorithm.  Assume the utility function   $\phi:\script{B}\rightarrow\mathbb{R}$ 
is smooth, concave, and  satisfies the assumptions of Section \ref{section:assumptions}.  Further, assume $\phi$ is 
\emph{$\alpha$-strongly concave}, meaning that: 
$\phi(\gamma) + \frac{\alpha}{2} \norm{\gamma}^2$ is also a concave function over $\gamma \in \script{B}$ (equivalently, $-\phi$ is an $\alpha$-strongly convex function).   Define $x^*$ as the (nonrandom) vector in the set $\overline{\Gamma}^*$ that 
corresponds to utility optimality for problem \eqref{eq:p1}-\eqref{eq:p2} (so that $\phi(x^*) = \phi^{opt}$). Let $\overline{\mu}[T] = \frac{1}{T}\sum_{t=0}^{T-1} \mu[t]$ be the (random) sample path time average over the first $T$ slots under the RUN algorithm. The \emph{mean square error} between $\overline{\mu}[T]$ and $x^*$ is: 
$$ \expect{\norm{\overline{\mu}[T] - x^*}^2} = \sum_{i=1}^n \expect{(\overline{\mu}_i[T] - x_i^*)^2} $$

\begin{thm} \label{thm:MSE}  If $\phi(\gamma)$ is $\alpha$-strongly concave over $\gamma \in \script{B}$, then for all $T>0$ we have:

a)  The RUN algorithm ensures
$$ \expect{\norm{\overline{\mu}[T] - x^*}^2} \leq \frac{G\norm{\mu^{max}}^2(1 + \log(T))}{\alpha T} $$

b) The EXP algorithm with parameter $\eta$ ensures 
$$  \expect{\norm{\overline{\mu}[T] - x^*}^2} \leq \frac{2(\phi^{opt} - \phi(0))}{\alpha \eta T} + \frac{\eta G\norm{\mu^{max}}^2}{\alpha}$$

c) The algorithm with stepsize $\eta_t=2/(t+2)$ ensures 
$$ \expect{\norm{\gamma[T] - x^*}^2} \leq \frac{4G\norm{\mu^{max}}^2}{\alpha(T+1)} $$
\end{thm} 
\begin{proof} 
Fix $T>0$. 
Recall that both the sample path time average $\overline{\mu}[T]$ and the optimal vector $x^*$ lie in the set $\script{B}$. The following inequality holds for any $\alpha$-strongly concave function evaluated at two points of its domain \cite{nesterov-book}:  
\begin{equation} \label{eq:start} 
\phi(\overline{\mu}[T]) \leq \phi(x^*) + \phi'(x^*)^{\transpose} (\overline{\mu}[T]-x^*) - \frac{\alpha}{2}\norm{\overline{\mu}[T]-x^*}^2 
\end{equation} 
where $\phi'(x^*)$ is a subgradient of $\phi$ at the point $x^*$. Taking expectations of both sides gives: 
$$ \expect{\phi(\overline{\mu}[T])} \leq \phi(x^*) + \phi'(x^*)^{\transpose}(\expect{\overline{\mu}[T]}-x^*) - \frac{\alpha}{2}\expect{\norm{\overline{\mu}[T]-x^*}^2} $$
Now note that $\expect{\overline{\mu}[T]}$ is a convex combination of points in the convex set 
$\overline{\Gamma}^*$ and hence 
lies in the set $\overline{\Gamma}^*$. Since $x^*$ maximizes the utility function $\phi$ over all other vectors in $\overline{\Gamma}^*$, 
the standard first order optimality condition requires: 
$$ \phi'(x^*)^{\transpose}(\expect{\overline{\mu}[T]}-x^*) \leq 0 $$
Substituting this inequality into the previous one gives: 
$$ \expect{\phi(\overline{\mu}[T])} \leq \phi(x^*) - \frac{\alpha}{2}\expect{\norm{\overline{\mu}[T]-x^*}^2} $$
Rearranging terms and using Theorem \ref{thm:RUN} yields the result of part (a), while using Theorem \ref{thm:EXP} yields the result of part (b).  Part (c) follows by a similar analysis that starts by comparing $\gamma[T]$ and $x^*$ in an inequality similar to \eqref{eq:start} (rather than comparing $\overline{\mu}[T]$ and $x^*$) and then using Theorem \ref{thm:new-stepsize}. 
\end{proof} 

The performance bound in the above theorem can be appreciated as follows: Recall that if $\{X_i\}_{i=0}^{\infty}$ is a sequence of independent and identically distributed (i.i.d.) random variables with finite mean and variance given by $m = \expect{X_i}$ and $\sigma^2 = Var(X_i)$, then the mean square error between the sample average $\frac{1}{T}\sum_{i=0}^{T-1} X_i$ and the mean $m$ is equal to: 
$$ \expect{\left(\frac{1}{T}\sum_{i=0}^{T-1} X_i - m\right)^2} = \frac{Var(X_0)}{T} $$
Hence, the mean square error is inversely proportional to the number of samples.  Theorem \ref{thm:MSE} shows that, under RUN, the mean square error between the sample path transmission rate and the optimal time averaged rate $x^*$ has a similar decay (differing only by a log factor).  This is remarkable because the network utility maximization problem involves joint estimation, learning, and control, and is much more complex than simply time averaging i.i.d. random variables. 

\section{A stochastic converse result}

This section provides a simple example of an opportunistic 
scheduling system, together with a smooth and strongly concave utility function, 
such that all statistics-unaware algorithms have a utility optimality 
gap that is at least $\Omega(1/t)$, where $t$ is the number
 of time steps.   This converse bound is close to the $O(\log(t)/t)$ optimality gap achievable by the RUN algorithm (as shown in
 the previous section).  
Hence, RUN is a statistics-unaware algorithm with 
an asymptotic convergence rate that is at most a logarithmic factor away from optimality.  

 \subsection{A 2-user system with ON/OFF channels} 
 
 Consider a 2-user system with an i.i.d. channel state process $\{S[t]\}_{t=0}^{\infty}$. Suppose there are only three possible channel state vectors,  so that $S[t] \in \{(ON, OFF), (ON, ON), (OFF, ON)\}$.  Every slot $t$, the network controller observes $S[t]$ and chooses to either transmit over exactly one channel that is currently ON, or to remain idle. The corresponding decision sets are: 
 \begin{align*}
 S[t] = (ON, OFF) &\implies (\mu_1(t), \mu_2(t)) \in \{(0,0), (1,0) \}\\
 S[t] = (ON, ON) &\implies (\mu_1(t), \mu_2(t)) \in \{(0,0), (1,0), (0,1)\} \\
 S[t] = (OFF, ON) &\implies (\mu_1(t), \mu_2(t)) \in \{(0,0), (0,1)\} 
 \end{align*}
 Define the utility function $\phi:[0,1]^2 \rightarrow \mathbb{R}$ by 
 $$ \phi(\gamma_1, \gamma_2) = \log(1 + \gamma_1) + \log(1+\gamma_2) $$
 It can be shown that $\phi$ is smooth and strongly concave over its domain.\footnote{The proportionally fair utility function $\log(\gamma_1)+\log(\gamma_2)$ could be similarly considered, although this has singularities at $\gamma_i=0$.}   Since $\phi$ is entrywise increasing, efficient algorithms should transmit whenever there is at least one ON channel. The only non-trivial decision is which channel to choose when $S[t] = (ON, ON)$.  Consider a particular \emph{statistics-unaware} algorithm $\pi$ that transmits whenever there is at least one ON channel, and if $S[t] \in (ON, ON)$ it chooses between the two transmission vectors $(1,0)$ and $(0,1)$ according to some (possibly randomized) policy.   Like the RUN, EXP, and DPP algorithms, the algorithm $\pi$ has no initial knowledge of the probability mass function for $S[t]$ and can only base decisions on current and past observations.   One can imagine that algorithm $\pi$ is chosen \emph{first}, then a probability mass function (PMF) for $S[t]$ is chosen by nature.  Nature is free to 
 choose a PMF under which policy $\pi$ performs poorly.  Consider two different PMFs, labeled PMF A and PMF B 
 in Table \ref{tab:1}.
 
 \begin{table}[htbp]
    \centering
    \begin{tabular}{@{} lcr @{}} 
       \multicolumn{2}{c}{} \\
            $S[t]$   & PMF A &  PMF B\\
       \midrule
       (ON, OFF)      & $3/4$ & $0$ \\
       (ON, ON)       & $1/4$  & $1/4$ \\
       (OFF, ON)     & $0$  & $3/4$ \\
       \bottomrule
    \end{tabular}
    \caption{Values for PMF A and PMF B.}
    \label{tab:1}
 \end{table}

 On slot $t=0$, the algorithm $\pi$  must have a contingency plan for choosing $(\mu_1[0], \mu_2[0])$ if it observes $S[0]=(ON, ON)$. 
Define: 
$$ \theta = P[(\mu_1[0], \mu_2[0]) = (1,0) | S[0]=(ON, ON)] $$
where this conditional probability $\theta$ is determined by the (potentially randomized) decision of algorithm $\pi$ on slot $0$, and is not connected to any past observations. 
In particular, the value of $\theta$ is determined before nature chooses the PMF.
 
  Below we show that, once the algorithm $\pi$ is chosen (which fixes the value of $\theta$), nature can choose a PMF 
 such that:  
  $$ \phi(\expect{\overline{\mu}_1[T]}, \expect{\overline{\mu}_2[T]})  \leq \phi^{opt} - \frac{1}{35 T} \quad, \forall T \in \{2, 3, 4, ...\} $$
 where the left-hand-side represents the utility achieved by algorithm $\pi$ over the first $T$ slots, and $\phi^{opt}$ is the optimal utility of the network under the PMF that was chosen by nature.  
 
\subsection{Case 1: $\theta \in [1/2, 1]$}  

\begin{figure}[htbp]
   \centering
   \includegraphics[height=1.5in]{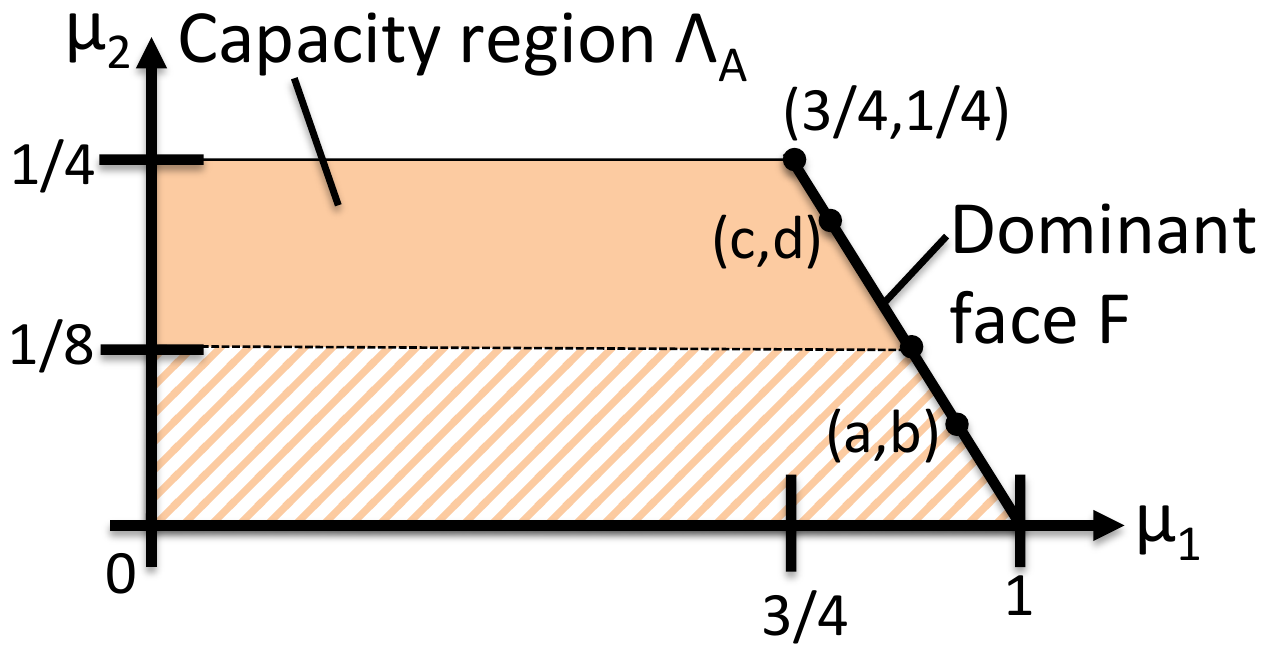} 
   \caption{The capacity region $\Lambda_A$ under PMF A.  All algorithms that transmit whenever possible have average rates that lie on the dominant face $F$. The point $(a,b)$ must lie in the intersection of the shaded region with the dominant
   face $F$.}
   \label{fig:PMFA}
\end{figure}

 Suppose $\theta \in [1/2,1]$.  Suppose nature chooses PMF A. Fix $T \in \{2, 3, 4, ...\}$.  
Define vectors $(a,b)$ and $(c,d)$ by 
 \begin{align}
 (a,b) &= \expect{(\mu_1[0], \mu_2[0])} \nonumber \\
 (c,d) &= \frac{1}{T-1}\sum_{t=1}^{T-1} \expect{(\mu_1[t], \mu_2[t])} \label{eq:cd} 
 \end{align}
 where the expectations are with respect to the random $S[t]$ channels that arise over time (which occur according to PMF A) and the possibly random decisions of policy $\pi$ in reaction to the observed channels. 
 We have: 
 \begin{equation} \label{eq:abcd} 
 (\expect{\overline{\mu}_1[T]}, \expect{\overline{\mu}_2[T]})  = \frac{1}{T}(a,b) + \frac{T-1}{T}(c,d)
 \end{equation} 
 Note that $(c,d)$ must be a point in the \emph{capacity region} $\Lambda_A$ that corresponds to PMF A, as shown in Fig. \ref{fig:PMFA} (this is because $\expect{(\mu_1[t], \mu_2[t])} \in \Lambda$ for all slots $t$, and so 
 $(c,d)$ defined in \eqref{eq:cd}  is a convex combination of points in the convex set $\Lambda$ and hence must also be in  $\Lambda$). 
 Define $F$ as the \emph{dominant face} of $\Lambda_A$, being the line segment in Fig. \ref{fig:PMFA} between 
 points $(3/4,1/4)$ and $(1,0)$.  Let $(\tilde{c}, \tilde{d})$ be a point on $F$ that is entrywise greater than or equal to $(c,d)$ (possibly being $(c,d)$ itself). 
 It can be shown that optimal utility is achieved at the corner point $(3/4, 1/4) \in \Lambda_A$, so that: 
 $$ \phi^{opt} = \log(1+ 3/4) + \log(1+1/4) $$
 Under PMF A, the point $(a,b) = \expect{(\mu_1[0], \mu_2[0])}$ satisfies: 
$$
 (a,b) = \underbrace{P[S[t]=(ON, OFF)]}_{3/4}(1,0) + \underbrace{P[S[t]=(ON,ON)]}_{1/4}[\theta(1,0) + (1-\theta)(0,1)]  $$
 That is, $(a,b) = \frac{1}{4}(3 + \theta, 1-\theta)$.   In particular, $a+b=1$, $(a,b) \in F$, 
 and since $\theta \in [1/2, 1]$ it holds that $b \leq 1/8$.  Thus, $(a,b)$ lies in the intersection of the shaded region of Fig. \ref{fig:PMFA} with the dominant face $F$. 
 Then, 
 \begin{align*}
 \phi\left(\expect{\overline{\mu}_1[T]}, \expect{\overline{\mu}_1[T]}\right) &\overset{(a)}{=} \log(1 + \frac{a}{T} + \frac{(T-1)c}{T}) + \log(1 + \frac{b}{T} + \frac{(T-1)d}{T}) \\
 &\overset{(b)}{\leq}  \log(1 + \frac{a}{T} + \frac{(T-1)\tilde{c}}{T}) + \log(1 + \frac{b}{T} + \frac{(T-1)\tilde{d}}{T})\\
  &\overset{(c)}{\leq}  \max_{(x,y) \in F}\left[\log(1 + \frac{a}{T} + \frac{(T-1)x}{T}) + \log(1 + \frac{b}{T} + \frac{(T-1)y}{T})\right]\\
 &\overset{(d)}{=} \log( 1+ \frac{a}{T} + \frac{(T-1)(3/4)}{T}) + \log(1 + \frac{b}{T} + \frac{(T-1)(1/4)}{T})\\
 &= \log(1 + 3/4 + \frac{(a-3/4)}{T}) + \log(1 + 1/4 + \frac{(b-1/4)}{T}) \\
 &\overset{(e)}{\leq} \log(1 + 3/4)  + \frac{a-3/4}{(1+3/4)T} + \log(1 + 1/4)  + \frac{b-1/4}{(1+1/4)T}\\
 &\overset{(f)}{=} \phi^{opt} - \frac{(1/4-b)(8/35)}{T} \\
 &\overset{(g)}{\leq} \phi^{opt} - \frac{1}{35 T} 
 \end{align*}
 where (a) holds by substituting \eqref{eq:abcd} into the utility function $\phi(\gamma_1, \gamma_2) = \log(1+\gamma_1)+\log(1+\gamma_2)$; (b) holds because $(\tilde{c}, \tilde{d})$ is entrywise greater than or equal to  $(c,d)$ and the utility function is entrywise increasing; 
 (c) holds because $(\tilde{c}, \tilde{d}) \in F$; (d) holds because the $(x,y)$ vector that maximizes the given expression over $F$ is $(x^*, y^*)=(3/4,1/4)$, which can be proven by observing that (i)  $(a,b) \in F$ and so for any $(x,y) \in F$ we have 
 $(a,b)/T + (x,y)(T-1)/T \in F$,  (ii) utility 
 increases as we move along the dominate face towards the corner point $(3/4,1/4)$, and  so the $(x,y)$ vector that maximizes
 the given expression over $F$ is $(3/4,1/4)$; (e)  holds because concavity of the function $\log(w+z)$ with respect to $z$ 
 implies $\log(w+z) \leq \log(w) + \frac{z}{w}$ for any real numbers $w,z$ that satisfy $w>0, w+z>0$; (f) holds because $a=1-b$; (g) holds because $b \leq 1/8$. 
 
 \subsection{Case 2: $\theta \in [0, 1/2)$} 
 
 \begin{figure}[htbp]
   \centering
   \includegraphics[width=2in]{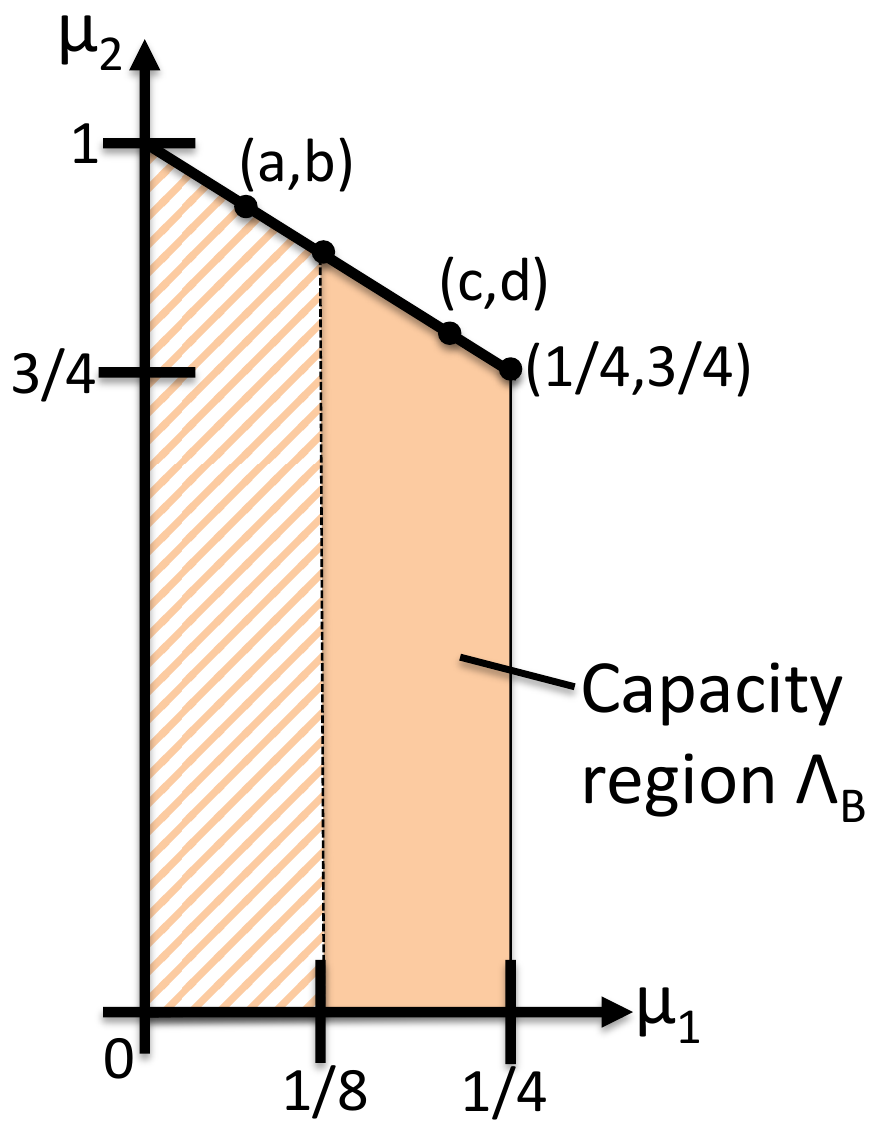} 
   \caption{The capacity region $\Lambda_B$ under PMF B. This is a symmetric flip of $\Lambda_A$.}
   \label{fig:PMFB}
\end{figure}
   
 Suppose $\theta \in [0, 1/2)$. However, now suppose nature chooses PMF B. The resulting capacity region $\Lambda_B$ is shown in Fig. \ref{fig:PMFB}.  Defining $(a,b)$ and $(c,d)$ as before, it can be shown that $(c,d) \in \Lambda_B$ and 
 $(a,b)$ is in the intersection of the shaded portion of $\Lambda_B$ with its dominant face (see Fig \ref{fig:PMFB}).  The situation is ``symmetric'' to that of Case 1 and a similar argument proves:
 $$  \phi\left(\expect{\overline{\mu}_1[T]}, \expect{\overline{\mu}_1[T]}\right)  \leq \phi^{opt} - \frac{1}{35T} $$ 
In particular, under either Case 1 or Case 2, a PMF can be chosen for which the optimality gap is at least $1/(35T)$.  It is impossible
for any statistics-unaware algorithm to ensure an optimality gap that decays faster than $1/(35T)$. 

\section{Scheduling in deterministic systems} \label{section:deterministic} 

Theorems \ref{thm:RUN}-\ref{thm:MSE} hold for general stochastic problems.  A special case of a 
stochastic system is a deterministic system where $\mu[t]$ is chosen from the same closed and bounded (possibly nonconvex)
set $\Gamma$ every slot $t$.   
In this deterministic case, the expectations in Theorems \ref{thm:RUN}-\ref{thm:MSE} 
can be removed (since all expectations are equal to their arguments with probability 1). 
If $\Gamma$ is a nonconvex set then utility optimality typically requires different points in $\Gamma$ to be selected with 
different fractions of time.  The implementation of the algorithm over time $\{0, 1, \ldots, T-1\}$ specifies how often 
each different rate vector should be chosen. 

\subsection{Deterministic RUN} 

In this deterministic case, Theorem \ref{thm:RUN} ensures the RUN algorithm deterministically yields 
$$ \phi(\overline{\mu}[T]) \geq \phi^{opt} - \frac{G\norm{\mu^{max}}^2(1+\log(T))}{2T} \quad, \forall T \in \{1, 2, 3, \ldots\} $$ 
where $\overline{\mu}[T] = \frac{1}{T}\sum_{t=0}^{T-1} \mu[t]$.  Further, if the utility function is additionally $\alpha$-strongly convex
then Theorem \ref{thm:MSE} proves that RUN gives: 
$$ \norm{\overline{\mu}[T] - x^*}^2 \leq \frac{G\norm{\mu^{max}}^2(1+\log(T))}{\alpha T} \quad, \forall T \in \{1, 2, 3, \ldots\} $$

This is useful for online implementation in the deterministic system.  It is also useful for \emph{offline computation}:  Fix $\epsilon>0$ and choose the smallest integer $T$ so that $\frac{G\norm{\mu^{max}}^2(1+\log(T))}{2T} \leq \epsilon$. Thus,  $T = \Theta(\log(1/\epsilon)/\epsilon)$. Run the algorithm over slots $t \in \{0, 1, \ldots, T-1\}$, observe what $\mu[t]$ vectors are chosen during this time, and define fractions of time for choosing each vector according to the fractions of time they are used over the interval $\{0, 1, \ldots, T-1\}$. This offline computation requires $\Theta(\log(1/\epsilon)/\epsilon)$ iterations.

\subsection{Deterministic $\eta_t = 2/(t+2)$}

The stepsize $\eta_t = 2/(t+2)$ in Theorem \ref{thm:new-stepsize} can be used to improve offline computation time using the unusual weighted average of Section \ref{section:unusual}. Indeed, Theorem \ref{thm:new-stepsize} for this deterministic context gives: 
$$ \phi(\gamma[T]) \geq \phi^{opt} - \frac{2G\norm{\mu^{max}}^2}{T+1} \quad, \forall T \in \{0, 1, 2, 3, \ldots\}  $$ 
Further, if the utility function is additionally $\alpha$-strongly convex then Theorem \ref{thm:MSE} proves 
$$ \norm{\gamma[T] - x^*}^2 \leq \frac{4G\norm{\mu^{max}}^2}{\alpha(T+1)}  \quad, \forall T \in \{0, 1, 2, 3 \ldots\}$$
This is useful for offline computation but requires the implemented $\mu[t]$ decisions to be \emph{reweighted} at the end of the run of $T+1$ slots.  Specifically, fix $T>0$ and run the algorithm over slots $t \in \{0, 1, 2, \ldots, T\}$.  Observe the resulting vectors $\{\mu[0], \mu[1], \ldots, \mu[T]\}$.  The vector $\gamma[T]$ is a convex combination of these vectors.   However, the convex combination must be computed according to the unusual weights shown in the example computations \eqref{eq:unusual1}-\eqref{eq:unusual4}.  Specifically, we have 
\begin{equation} 
\gamma[T] = \sum_{t=0}^{T} w[t]\mu[t] \label{eq:fractions-of-time} 
\end{equation} 
where the weights $w[t]$ are nonnegative and satisfy $\sum_{t=0}^T w[t]=1$.  The value $w[t]$ determines the correct fraction of time to use vector $\mu[t] \in \Gamma$ in order to achieve the time average vector $\gamma[T]$ in \eqref{eq:fractions-of-time}.  The weights $w[t]$ 
 can be determined by the following iterative procedure that grows a vector $\vec{M}_k$ by one dimension on each step:  
\begin{itemize}
\item Define $\vec{M}_0 = [1]$.  
\item At step $k \in \{1, \ldots, T\}$, define $\vec{M}_k = \left[\left(\frac{k}{k+2}\right)\vec{M}_{k-1}; \frac{2}{k+2}\right]$. 
\end{itemize}
At step $T$, the vector $\vec{M}_T$ has $T+1$ dimensions with components given by the desired $w[t]$ values: 
$$ \vec{M}_T = [w[0]; w[1]; \ldots; w[T]] $$
For example, the first few steps of this procedure give the following weights that correspond to \eqref{eq:unusual1}-\eqref{eq:unusual4}: 
\begin{align*}
\vec{M}_0 &= \left[1\right]\\
\vec{M}_1 &= \left[\frac{1}{3}; \frac{2}{3}\right]\\
\vec{M}_2 &=\left[\frac{1}{6}; \frac{1}{3}; \frac{1}{2}\right]\\
\vec{M}_3 &= \left[\frac{1}{10};\frac{1}{5}; \frac{3}{10};\frac{2}{5}\right]
\end{align*}

\section{Conclusion} 

This paper considers stochastic utility maximization for opportunistic scheduling systems. 
It shows that  all  statistics-unaware algorithms incur error that is at least $\Omega(1/t)$ 
after $t$ slots.  
A stochastic variation of the 
Frank-Wolfe algorithm called RUN is shown to achieve error that decays like $O(\log(t)/t)$.   Unfortunately, RUN uses a vanishing
stepsize and has no adaptation capabilities.  The EXP algorithm uses a fixed stepsize for better
adaptation but worse convergence time.  Specifically, EXP is shown to achieve an $O(\epsilon)$-approximation 
with convergence and adaptation times of $O(1/\epsilon^2)$, similar to the DPP algorithm. 
A variation of Franke-Wolfe that uses a (vanishing) stepsize different from RUN and EXP is shown to compute a random vector whose expectation is within $O(1/t)$ of optimal utility (without a $\log$ factor), although this random vector does not correspond to the time average transmission rates used over the first $t$ slots. 
 In terms of convergence time, 
it is unclear how the gap can be closed between  the $O(\log(t)/t)$ achievability bound and the $\Omega(1/t)$ converse.  It is also unclear if 
algorithms exist that improve both convergence and adaptation times beyond $O(1/\epsilon^2)$. 

\section*{Appendix} 

\begin{lem} Fix $\eta_t = \eta$ for all $t \in \{0, 1, 2, \ldots\}$ and for some $\eta \in (0,1)$. 
The iteration \eqref{eq:weighted-average} ensures: 
$$ \gamma[t] = (1-\eta)^{t+1} \gamma[-1] + \sum_{k=0}^t \eta (1-\eta)^{k}\mu[t-k]  \quad , \forall 
t \in \{0, 1, 2, \ldots\} $$
In particular, since $\gamma[-1]=0 \leq \mu[0]$, we have that $\gamma[t]$ is entrywise less
than or equal to the following exponentially weighted average of the $\{\mu[t]\}$ process: 
$$ \gamma[t] \leq (1-\eta)^{t+1} \mu[0] + \sum_{k=0}^t \eta (1-\eta)^{k}\mu[t-k] \quad , \forall 
t \in \{0, 1, 2, \ldots\} $$
\end{lem} 
\begin{proof} 
The proof follows by induction on \eqref{eq:weighted-average}. 
\end{proof} 

\begin{lem} \label{lem:dominate}  The EXP algorithm ensures $\expect{\phi(\gamma[t])} \leq \phi^{opt}$ for all $t \in \{0, 1, 2, \ldots\}$. 
\end{lem} 
\begin{proof}  Fix $t \in \{0, 1, 2, \ldots\}$. 
From the above lemma and the entrywise nondecreasing property of $\phi$ we have: 
$$ \phi(\gamma[t]) \leq \phi\left(  (1-\eta)^{t+1} \mu[0] + \sum_{k=0}^t \eta (1-\eta)^{k}\mu[t-k] \right) $$
Taking expectations of both sides and using Jensen's inequality in the right-hand-side gives: 
\begin{align*} 
\expect{\phi(\gamma[t])} &\leq \phi\left(  (1-\eta)^{t+1} \expect{\mu[0]} + \sum_{k=0}^t \eta (1-\eta)^{k}\expect{\mu[t-k]}\right) \\
&=  \phi\left(  (1-\eta)^{t+1}x_0 + \sum_{k=0}^t \eta (1-\eta)^{k}x_{t-k}\right) 
\end{align*} 
where we define $x_t = \expect{\mu[t]}$ for each $t \in \{0, 1, 2, \ldots\}$.  Define the vector $y$ by: 
$$ y =  (1-\eta)^{t+1}x_0 + \sum_{k=0}^t \eta (1-\eta)^{k}x_{t-k} $$
so that we have $\expect{\phi(\gamma[t])} \leq \phi(y)$. 
Note that $x_t$ is a one-shot expectation 
of $\mu[t]$ on slot $t$, and hence must lie in the set $\Gamma^*$ (since all expectations that can be achieved on slot $t$ can also be achieved on slot $0$).  Hence, the vector $y$ is a convex combination 
of vectors in the convex set $\Gamma^*$, and so $y \in \Gamma^* \subseteq \overline{\Gamma}^*$.  Thus: 
$$
\expect{\phi(\gamma[t])} \leq \phi(y) 
\leq \max_{x \in \overline{\Gamma}^*} \phi(x) =\phi^{opt}
$$
where the last equality follows by  \eqref{eq:optimality}.  
\end{proof}

\bibliographystyle{unsrt}
\bibliography{refs}
\end{document}